\documentclass{amsart}
\usepackage{amsmath, amssymb}
\usepackage{url, tikz-cd}
\usepackage[mathscr]{euscript}

\newtheorem{innercustomthm}{Theorem}
\newenvironment{customthm}[1]
  {\renewcommand\theinnercustomthm{#1}\innercustomthm}
  {\endinnercustomthm}
  
\newtheorem{theorem}{Theorem}[section]
\newtheorem{lemma}[theorem]{Lemma}
\newtheorem{corollary}[theorem]{Corollary}
\newtheorem*{corollary*}{Corollary}

\newtheorem{proposition}[theorem]{Proposition}

\theoremstyle{definition}
\newtheorem{example}[theorem]{Example}

\newtheorem{remark}[theorem]{Remark}
\newtheorem*{remark*}{Remark}
\newtheorem{definition}[theorem]{Definition}
\newtheorem*{definition*}{Definition}

\newtheorem{exam}{Example}

\def \U {\mathcal U}
\def \C {\mathcal C}
\def \PP {\mathbb P}

\def\PGL{\operatorname{PGL}}
\def\DCF{\operatorname{DCF}}

\def\ccm{\operatorname{CCM}}

\def\dcl{\operatorname{dcl}}
\def\acl{\operatorname{acl}}
\def\alg{\operatorname{alg}}

\def\tp{\operatorname{tp}}

\def\RM{\operatorname{RM}}
\def\aut{\operatorname{Aut}}
\def\Gal{\operatorname{Gal}}

\def\C{\mathcal C}

\def\nmdeg{\operatorname{nmdeg}}

\def\Ind#1#2{#1\setbox0=\hbox{$#1x$}\kern\wd0\hbox to 0pt{\hss$#1\mid$\hss}
\lower.9\ht0\hbox to 0pt{\hss$#1\smile$\hss}\kern\wd0}
\def\ind{\mathop{\mathpalette\Ind{}}}
\def\Notind#1#2{#1\setbox0=\hbox{$#1x$}\kern\wd0\hbox to 0pt{\mathchardef
\nn=12854\hss$#1\nn$\kern1.4\wd0\hss}\hbox to
0pt{\hss$#1\mid$\hss}\lower.9\ht0 \hbox to
0pt{\hss$#1\smile$\hss}\kern\wd0}
\def\nind{\mathop{\mathpalette\Notind{}}}

\newcommand{\m}{\mathbb }
\newcommand{\mc}{\mathcal }

\title{When any three solutions are independent}

\author{James Freitag}
\address{James Freitag\\
University of Illinois Chicago\\ 
Department of Mathematics, Statistics,
and Computer Science\\ 
851 S. Morgan Street\\
Chicago, IL, 60607-7045\\
USA}
\email{jfreitag@uic.edu}

\author{R\'emi Jaoui}
\address{R\'emi Jaoui\\
Albert-Ludwigs Universität Freiburg\\
Abteilung für Mathematische Logik, Mathematisches Institut\\ Ernst-Zermelo-Straße 1, D-79104 Freiburg\\ Germany.}
\email{remi.jaoui@math.uni-freiburg.de}

\author{Rahim Moosa}
\address{Rahim Moosa\\
University of Waterloo\\
Department of Pure Mathematics\\
200 University Avenue West\\
Waterloo, Ontario \  N2L 3G1\\
Canada}
\email{rmoosa@uwaterloo.ca}

\subjclass[2020]{03C45, 12H05, 11J81, 32J27}
\keywords{Geometric stability theory, differentially closed fields, algebraic differential equations, transcendence, $D$-varieties, compact K\"ahler manifolds}

\thanks{The first author was partially supported by NSF grant DMS-1700095 and NSF CAREER award 1945251.
The second author was partially supported by the ANR-DFG program GeoMod (Project number 2100310201).
The third author was partially supported by an NSERC Discovery Grant.
The first and third authors would also like to thank the hospitality of the Fields Institute in Toronto during the 2021 thematic programme on Trends in Pure and Applied Model Theory, where much of this work was carried out.
}

\date{\today}

\begin{document}

\begin{abstract}
Given an algebraic differential equation of order greater than one, it is shown that if there is any nontrivial algebraic relation amongst any number of distinct nonalgebraic solutions, along with their derivatives, then there is already such a relation between three solutions.
In the autonomous situation when the equation is over constant parameters the assumption that the order be greater than one can be dropped, and a nontrivial algebraic relation exists already between two solutions.
These theorems are deduced as an application of the following model-theoretic result: Suppose $p$ is a stationary nonalgebraic type in the theory of differentially closed fields of characteristic zero; if any three distinct realisations of $p$ are independent then $p$ is minimal.
If the type is over the constants then minimality (and complete disintegratedness) already follow from knowing that any two realisations are independent.
An algebro-geometric formulation in terms of $D$-varieties is given.
The same methods yield also an analogous statement about families of compact K\"ahler manifolds.
\end{abstract}

\maketitle

\setcounter{tocdepth}{1}
\tableofcontents

\section{Introduction}

\noindent
Understanding algebraic relations between solutions of a differential equation is, perhaps along with understanding solutions of a particular form, the central problem in algebraic differential equations.
Characterizing algebraic relations between solutions is an important output of various approaches to differential Galois theory~\cite{kolchin1968algebraic, singer1988algebraic, dreyfus2018hypertranscendence}, while work using the model theory of differential fields frequently applies techniques from stability theory to obtain conclusions in this vein~\cite{nagloo2014algebraic, nagloo2017algebraic, jaoui2019generic}.
A number of recent transcendence results for analytic functions (equivalently, problems of \emph{bi-algebraic geometry}~\cite{klingler2018bi}) have played an important role in diophantine geometry and have natural interpretations in terms of characterizing algebraic relations between solutions of differential equations and their derivatives~\cite{pila2016ax, freitag2017strong, casale2020ax}. 

To state our main theorem let us fix a differential field $(k,\delta)$ of characteristic zero, as well as an order $n$ algebraic differential equation
\begin{equation}
\label{ode}
P(y,\delta y,\dots,\delta^{(n)}y)=0,
\end{equation}
where $P\in k[X_0,\dots,X_n]$ is irreducible.
For each $m\geq 1$, consider the following condition:
\begin{itemize}
\item[$(C_m)$]
For any $m$ distinct solutions $a_1,\dots,a_m\notin k^{\alg}$ of~(\ref{ode}), the sequence
$$(\delta^{(i)}a_j:i=0,\dots,n-1,j=1,\dots,m)$$
is algebraically independent over $k$.
\end{itemize}
In~\cite{nmdeg}, as a byproduct of working on a related problem, the first and third authors showed that as long as $n>1$, one always has $(C_{n+2})\implies(C_m)$ for all $m$.
We can now do much better:

\begin{customthm}{A}
\label{A}
\begin{itemize}
\item[(a)]
If $n>1$ then $(C_3)\implies(C_m)$ for all $m$.
\item[(b)]
If $\delta$ is trivial on $k$ then $(C_2)\implies(C_m)$ for all $m$.
\end{itemize}
\end{customthm}

Algebraic differential equations over constant parameters, as is assumed in part~(b), often go by the name {\em autonomous}.

Part~(a) of the theorem is sharp in both natural ways:  there are (non-autonomous) equations in every order greater than~$1$ that satisfy $(C_2)$ but not $(C_3)$, and there are (non-autonomous) order $1$ equations that satisfy $(C_3)$ but not $(C_4)$.
See Examples~\ref{exa} and~\ref{exb} below, respectively.
However, in order~$1$ we do have $(C_4)\implies(C_m)$ for all $m$, as we will explain later.
Part~(b) is also sharp, for example $\delta y-1=0$ is an autonomous equation satisfying $(C_1)$ but not $(C_2)$.

Let us consider some specific examples that have been studied recently.

\begin{exam}
In~\cite{jaoui2019generic}, the second author studied the special case of autonomous equations of order two and degree at least 3 where the coefficients of $P$ form an algebraically independent set.
He showed that the equation is then {\em strongly minimal and geometrically trivial}, from which it follows that $(C_2)\implies(C_m)$.
\end{exam}

\begin{exam}
Consider a Painlev\'e equation from the families $P_{II}$ through $P_{V}$ with generic coefficients.
For instance, one might take $y''=2y^3+ty+\pi$ over the differential field $(\m C(t),\frac{d}{dt})$.
Nagloo and Pillay \cite{nagloo2017algebraic} show that $(C_m)$ holds of all $m$.
\end{exam}

\begin{exam}
Consider the following autonomous equation satisfied by the $j$-function 
$$S(y) + \frac{y^2 - 1968 y + 2654208}{2 y^2 (y - 1728)^2} (y')^2=0,$$ where $S$ is the Schwarzian derivative: 
$S(y)=\left( \frac{y''}{y'}\right) '-\frac{1}{2} \left(\frac{y''}{y'}\right)^2$.
It follows from results of \cite{freitag2017strong} that $(C_2)$ fails of this equation.
In fact, if $y_1,y_2$ are nonalgebraic solutions then $(y_1,y_1', y_1'',y_2,y_2',y_2'')$ are algebraically dependent if and only if a modular polynomial vanishes at $(y_1,y_2 )$. 
A generalisation of this equation and setting is considered in \cite{casale2020ax}, where the $j$-function is replaced by the uniformisation function of the quotient of the upper half plane by an arbitrary genus zero Fuchsian group $\Gamma$ of the first kind.
For such equations, it is shown that $(C_2)$ fails if and only if $(C_m)$ fails for some $m$, if and only if $\Gamma$ is a proper subgroup of its commensurator. 
This analysis of the algebraic relations amongst the solutions and their derivatives, in~\cite{casale2020ax}, is key to proving the Ax-Lindemann-Weierstrass theorem for uniformising functions of Fuchsian groups and answering an old open question of Painlev\'e (1895).
Ax-Lindemann-Weierstrass results for the uniformizing functions of Shimura varieties are an essential component of the resolution of a number of recent diophantine conjectures and can naturally be phrased in terms of algebraic relations between solutions of differential equations (see \cite{klingler2018bi} and the references therein).
\end{exam}

In fact, our proof of Theorem~A works more generally with finite-dimensional systems of algebraic differential equations in several variables.
In the autonomous case this is best expressed in terms of algebraic vector fields; namely algebraic varieties $V$ equipped with a polynomial section~$s$ to the tangent bundle.
To study the algebraic relations between the solutions of the corresponding system of equations amounts to looking at subvarieties of cartesian powers of $V$ that are invariant under the vector field.
(Here we put on $V^m$ the canonical vector field induced by $s$, obtained by identifying $T(V^m)$ with the $m$-fold fibred product of $TV$ with itself.)
We can formulate our main theorem, in the autonomous case, as follows:

\begin{customthm}{B}
\label{B}
Suppose $(V,s)$ is an algebraic vector field over a field $k$ of characteristic zero.
If $V^2$ admits no proper invariant subvarieties over $k$ projecting dominantly onto each co-ordinate, other than the diagonal, then the same holds of $V^m$ for all $m$, that is, the various diagonals in $V^m$ are the only invariant subvarieties over $k$ projecting dominantly onto each co-ordinate.
\end{customthm}

In the non-autonomous setting one works with algebraic ``$D$-varieties" in the sense of Buium~\cite{Buium1} rather than algebraic vector fields; where the tangent bundle is replaced by the {\em prolongation}, a natural torsor of the tangent bundle that takes into account the twisting induced by~$\delta$ on $k$.
See~$\S$\ref{sectdvar} below for a discussion.
We have an analogue of Theorem~\ref{B} in that setting, but, as one might expect from part~(a) of Theorem~\ref{A}, we have to assume that $\dim V>1$ and the assumption on invariant subvarieties has to be made about $V^3$ in order to conclude it for all $V^m$.
See Corollary~\ref{dvarcor} below for the general statement that includes both the autonomous and non-autonomous setting.

There is an apparent discrepancy between Theorems~\ref{A} and~\ref{B}.
The latter is really about relations among {\em generic} solutions to the differential equation -- this is the meaning of only considering invariant subvarieties that project {\em dominantly} on each co-ordinate -- while Theorem~\ref{A} seems to be stronger, making conclusions about all nonalgebraic solutions.
In fact, however, $(C_1)$ already rules out the possibility of lower order equations, so that all nonalgebraic solutions are generic.

These theorems are applications of a more fundamental theorem about the model theory of differentially closed fields ($\DCF_0$), which we now state.
We work model-theoretically, assuming familiarity with geometric stability theory as well as the particular case of $\DCF_0$.
We suggest~\cite{gst} as a general reference for the former, and~\cite{Marker96modeltheory} as an introduction to the latter.

The following combines the main clauses of Theorems~\ref{c3min} and~\ref{c2min} below:
\begin{customthm}{C}
\label{C}
Suppose $p$ is a stationary nonalgebraic type in $\DCF_0$.
\begin{itemize}
\item[(a)]
If any three distinct realisations of $p$ are independent then $p$ is minimal.
\item[(b)]
If $p$ is over the constants, and any two distinct realisations of $p$ are independent, then $p$ is minimal.
\end{itemize}
\end{customthm}
Once we have a minimal generic type we can apply the Zilber trichotomy in differentially closed fields and deduce Theorems~\ref{A} and~\ref{B} quite readily.
The proof of minimality itself goes via reducing to the case when $p$ is internal to the constants and with a {\em binding group} that acts $3$-transitively, in the case of part~(a).
As the binding group is an algebraic group action, a theorem of Knop~\cite{knop1983mehrfach} implying that the only $3$-transitive algebraic group action is that of $\PGL_2$ on $\mathbb P$ now applies, and settles the matter.
For the autonomous case, namely part~(b), we do not require Knop's classification, but make use of a little differential Galois theory implying that the binding group cannot be centerless.

Other than standard facts from geometric stability theory and the model theory of differentially closed fields, as well as the theorem of Knop just mentioned, the paper is self-contained.
In particular, the results of~\cite{nmdeg} are not used here, though the approach taken there does inform what we do.

The proof of Theorem~\ref{C}(a) works in any totally transcendental theory where one has the following strong form of the Zilber dichotomy: every non locally modular minimal type over $A$ is nonorthogonal to an $A$-definable pure algebraically closed field.
So, for example, we could just as easily have worked with several commuting derivations in partial differentially closed fields ($\DCF_{0,m}$).
Or in the theory $\ccm$ of compact complex manifolds.
Indeed, concerning the latter, a short final section is dedicated to articulating the result as a theorem in bimeromorphic geometry that may be of independent interest, see Theorem~\ref{ccm3min} below.

\medskip
{\bf A note on authorship.}
The first and third authors initially circulated a version of this paper that did not consider the autonomous case separately.
So, for example, it did not include part~(b) of Theorems~\ref{A} and~\ref{C}.
The second author, upon reading the preprint, realised that the method would, with the addition of some differential Galois theory, yield stronger results in the autonomous case.
The current paper incorporates his improvements.

\medskip
Throughout, except when otherwise specified, we work in a sufficiently saturated model $\mathcal U\models\DCF_0$ with field of constants $\C$.

\bigskip
\section{Generic transitivity and degree of nonminimality}

\noindent
Recall from~\cite{BC2008} that a definable group action $(G,S)$ of finite Morley rank is \emph{generically transitive} if the action of $G$ on $S$ admits an orbit $\mc O$ such that
$$\RM(S\setminus \mc O)<\RM(S).$$
Assuming that $\operatorname{dM}(S)=1$, this is equivalent to asking that any two generic elements of $S$ over $A$ are conjugate.

Actually, we are interested in the multiple-transitivity version:

\begin{definition}
A finite rank definable group action $(G,S)$ is {\em generically $k$-transitive} if the diagonal action $(G,S^k)$ is generically transitive.
\end{definition}

Generic $k$-transitivity, and the Borovik-Cherlin conjecture about such group actions (articulated in~\cite{BC2008}), played an important role in recent work~\cite{nmdeg} on bounding the size of a witness to the nonminimality of a type in $\DCF_0$.
Recall:

\begin{definition}
The {\em degree of nonminimality} of $p\in S(A)$ is the least $d$ such that there are distinct realisations $a_1,\dots,a_d$ of $p$ and a nonalgebraic forking extension of $p$ to $Aa_1\cdots a_d$.
\end{definition}

The main result of~\cite{nmdeg} was that $\nmdeg(p)\leq U(p)+1$.
One step in that proof was the following observation, which we repeat here for the sake of completeness.

\begin{proposition}
\label{gendtran}
Suppose $p$ is a stationary type with $U(p)>1$ that is $\C$-internal and weakly $\C$-orthogonal.
Let $d:=\nmdeg(p)$.
Then the binding group of $p$ acts transitively and generically $d$-transitively on $p(\U)$.
\end{proposition}

\begin{proof}
Denote the binding group by $G:=\aut_A(p/\C)$.
It acts definably and faithfully on $S:=p(\U)$.
Weak orthogonality implies that the action is transitive and $p$ is isolated.
So $(G,S)$ is a definable homogeneous space.
To show that the action is generically $d$-transitive is precisely to show that $G$ acts transitively on $p^{(d)}$.
As $G$ is also the binding group of the $\C$-internal type $p^{(d)}$, it suffices to show that $p^{(d)}$ is weakly orthogonal to $\C$.

Let $r>1$ be least such that $p^{(r)}$ is not weakly orthogonal to $\C$.
(This exists by $\C$-internality.)
Fix $\overline a=(a_1,\dots,a_{r-1})\models p^{(r-1)}$ and let $q$ be the nonforking extension of $p$ to $A\overline a$.
Then $q$ is not weakly orthogonal to $\C$.
It follows (using elimimination of imaginaries for the structure induced on $\C$; namely that of a pure algebraically closed field) that there is $a\models q$ and a constant $c\in\C$ such that $c\in\dcl(A\overline a a)\setminus\acl(A\overline a)$.
Write $c=f(a)$ where $f$ is a partial $A\overline a$-definable function from $S$ to $\C$.
The fibre $f^{-1}(c)$ is infinite as $U(q)>1$, and so we may assume that all the fibres are infinite.
Let $c_0\in\operatorname{Im}(f)\cap\mathbb Q$, which exists as $\operatorname{Im}(f)$ is an infinite and hence cofinite subset of $\C$.
Let $a_0$ be generic in $f^{-1}(c_0)$ over $A\overline a$.
Then $U(a_0/A\overline a)>0$ as $f^{-1}(c_0)$ is infinite.
On the other hand, $U(a_0/A\overline a)<U(p)$, as otherwise $a_0\models q$, which is impossible as $f(a)\notin\acl(A\overline a)$ while $f(a_0)\in\mathbb Q=\dcl(\emptyset)$.
Hence $\tp(a_0/A\overline a)$ is a nonalgebraic forking extension of $p$ witnessing that $d\leq r-1$.
By minimal choice of $r$, it follows that $p^{(d)}$ is weakly $\C$-orthogonal, as desired.
\end{proof}

\bigskip
\section{Distinct solutions being independent}

\noindent
We are interested in the following condition on a stationary type $p\in S(A)$.

\begin{definition}[The condition $D_m$]
Given $m>1$, let us say that {\em $p\in S(A)$ satisfies~$D_m$} if every $m$-tuple of distinct realisations of $p$ is independent over $A$.
\end{definition}

We say that $p$ is {\em completely disintegrated} if $D_m$ holds for all $m>1$.

We begin with a few straightforward consequences of $D_2$.

\begin{lemma}
Every type satisfying $D_2$ is of finite rank.
\end{lemma}

\begin{proof}
As some co-ordinate of any realisation of an infinite rank type must be $\delta$-transcendental, it suffices to observe that the $1$-type of a $\delta$-transcendental element does not satisfy $D_2$.
Indeed, if $a$ is $\delta$-transcendental over a differential field $k$ then so is $\delta(a)$, which is distinct for $a$, and $a\nind_k\delta a$.
\end{proof}

\begin{lemma}
\label{notfinitecover}
No type satisfying $D_2$ can be a nontrivial finite cover.
That is, if $p=\tp(a/A)$ is a nonalgebraic stationary type satisfying $D_2$, and $a$ is interalgebraic with $b$ over $A$, then $a\in\dcl(Ab)$.
\end{lemma}

\begin{proof}
If $a\notin\dcl(Ab)$ then we can find $a'\neq a$ such that $\tp(a'/Ab)=\tp(a/Ab)$.
By $D_2$ we must have $a'\ind_Aa$.
But $a\in\acl(Ab)=\acl(Aa')$, which forces $a\in\acl(A)$ and contradicts nonalgebraicity of $p$.
\end{proof}

\begin{lemma}
\label{noweakfibrations}
If $p=\tp(a/A)$ is a nonalgebraic stationary type satisfying $D_2$ then it satisfies the following form of exchange: if $b\in\acl(Aa)\setminus\acl(A)$ then $a\in\dcl(Ab)$.
\end{lemma}

\begin{proof}
Note that the conclusion is a strengthening of the property called {\em admitting no proper fibrations} in~\cite{moosa2014some}.
In any case, by Lemma~\ref{notfinitecover}, it suffices to show that $a\in\acl(Ab)$.
Suppose $a\notin\acl(Ab)$, and seek a contradiction.
It follows that if $a'$ realises the nonforking extension of $\tp(a/Ab)$ to $Aa$ then $a\neq a'$.
On the other hand, using the Lascar inequality, $U(a/Aa')=U(a/Ab)<U(a/A)$ as $b\in\acl(Aa)\setminus\acl(A)$.
So $a$ and $a'$ are dependent over $A$, contradicting $D_2$.
\end{proof}

\begin{proposition}
\label{c2internal}
Every stationary type of $U$-rank $>1$ satisfying $D_2$ is internal and weakly orthogonal to the constants.
\end{proposition}

\begin{proof}
Suppose $p=\tp(a/A)$ is stationary, nonminimal, and satisfies $D_2$.
As $p$ is of finite rank it is nonorthogonal to some minimal type $r \in S(B)$ for some $B\supseteq A$.
We first show that $r$ is not locally modular.
Indeed, in~\cite[Proposition~2.3]{moosa2014some}, it is shown that, because $p$ has no proper fibrations (Lemma~\ref{noweakfibrations}), if $r$ is locally modular than in fact $p$ is interalgebraic with some $q^{(k)}$ where $q\in S(A)$ is a locally modular minimal type and $k\geq 1$.
Now, $k=1$ is impossible by the assumption that $U(p)>1$.
But $k>1$ is also impossible: taking $(b_1,\dots,b_k)\models q^{(k)}$, we obtain
$$b_1\in\acl(Ab_1\cdots b_k)\setminus\acl(A)=\acl(Aa)\setminus\acl(A)$$
while
$$\acl(Aa)=\acl(Ab_1\cdots b_k)\not\subseteq\acl(Ab_1)$$
contradicting Lemma~\ref{noweakfibrations}.

We have shown that $p$ is nonorthogonal to a non locally modular minimal type.
In $\DCF_0$ this means that $p$ is nonorthogonal to the constants $\C$.
It follows by~\cite[Corollary~7.4.6]{gst} that there is $c\in\dcl(Aa)\setminus\acl(A)$ such that $q:=\tp(c/A)$ is $\C$-internal.
By Lemma~\ref{noweakfibrations}, $a\in\dcl(Ac)$.
That is, $p$ and $q$ are interdefinable, and hence $p$ is also $\C$-internal.

If $p$ were not weakly orthogonal to $\C$ then there would be $c\in\C$ such that $c\in\dcl(Aa)\setminus\acl(A)$.
Lemma~\ref{noweakfibrations} would then imply that $a\in\dcl(Ac)$, contradicting nonminimality of $p$ (and also $D_2$).
\end{proof}

Here is our main theorem.

\begin{theorem}
\label{c3min}
Suppose $p$ is a stationary nonalgebraic type satisfying~$D_3$.
Then $p$ is minimal.
Moreover, there are in this case exactly two possibilities:
\begin{itemize}
\item[(i)]
Either $p$ is completely disintegrated, or,
\item[(ii)]
$p$ is nonorthogonal to the constants and $D_4$ fails.
\end{itemize}
\end{theorem}

\begin{proof}
We assume $p\in S(A)$ is not minimal and seek a contradiction.
By Proposition~\ref{c2internal} we have that $p$ is $\C$-internal and weakly $\C$-orthogonal.
Denote the binding group by $G:=\aut_A(p/\C)$ and set $S:=p(\U)$.
Proposition~\ref{gendtran} tells us that $(G,S)$ is a generically $d$-transitive definable homogeneous space, where~$d$ is the degree of nonminimality of $p$.

Let $a_1,\dots,a_d$ be distinct realisations of $p$ and $q$ a nonalgebraic forking extension of $p$ to $Aa_1\cdots a_d$.
Let $a_{d+1}\models q$.
Then $a_1,\dots,a_{d+1}$ witnesses the failure of $D_{d+1}$.
Since $D_3$ holds, we must have $d\geq 3$.
In particular, $G$ acts generically $3$-transitively on $S$.
But $D_3$ tells us that every tuple of three distinct elements is generic.
Hence $(G,S)$ is outright $3$-transitive.

As $(G,S)$ is the binding group action of a $\C$-internal type, it is definably isomorphic (over possibly additional parameters) with a definable homogeneous space in the induced structure on $\C$.
That induced structure being a pure algebraically closed field, and using the Weil group-chunk theorem, we have that $(G,S)$ is definably isomorphic to an algebraic group action $(\overline G,\overline S)$ in the constants.
So $(\overline G,\overline S)$ is $3$-transitive.
But the only $3$-transitive algebraic group action is $\PGL_2$ acting naturally on $\PP$.
Indeed, Knop~\cite[Satz~2]{knop1983mehrfach} classifies the (faithful) $2$-transitive actions of algebraic groups as being of only two kinds:
\begin{itemize} 
\item The standard action of $\PGL_{n+1}$ on $\m P^n.$ 
\item Certain subgroups of the affine group of transformations on a vector space. 
\end{itemize}
In both cases, the action preserves colinearity of three points, and hence the only possibility for $3$-transitivity is if we are in the first case and $n=1$.
So $(\overline G,\overline S)$ is isomorphic to $(\PGL_2(\C),\PP(\C))$.
This contradicts nonminimality of $p$.

We have proven that $p$ is minimal.
The ``moreover" clause now follows by applying the Zilber trichotomy as it is manifested in $\DCF_0$.
That trichotomy says that exactly one of the following must hold:
\begin{itemize}
\item[(I)]
$p$ is geometrically trivial;
\item[(II)]
$p$ is nonorthogonal to the generic type of the Manin kernel $\mathcal G$ of a simple abelian variety over $\acl(A)$ that does not descend to the constants; or,
\item[(III)]
$p$ is nonorthogonal to the constants.
\end{itemize}
This is originally an unpublished theorem of Hrushovski and Sokolovic~\cite{HrSo}.
Published references and explanations can be found in various places, see for example the discussion around Fact~4.1 of~\cite{isolated}.

We first point out that $D_3$ rules out case~(II).
Indeed, assuming we are in case~(II), and
setting $q$ to be the nonforking extension of $p$ to $Aa$, it would follow that $q$ is interalgebraic with the generic type $r$ of $\mathcal G$ over $Aa$.
This is because both~$q$ and~$r$ would be modular minimal types, and for such types weak orthogonality implies orthogonality (see~\cite[Corollary~2.5.5]{gst}).
As $p$ satisfies $D_3$, $q$ satisfies $D_2$, and hence Lemma~\ref{noweakfibrations} implies that $r$ is a finite cover of~$q$, say $\pi:r\to q$.
Given $g\models r$, since $\mathcal G$ has finite $n$-torsion, $ng$ also realises $r$ for all $n$, and hence we can find $n>1$ such that $\pi(g)\neq\pi(ng)$.
But these two elements are distinct realisations of $q$ that are dependent over $Aa$.
It follows that $\{a,\pi(g),\pi(ng)\}$ witnesses the failure of $p$ to satisfy~$D_3$.

We must, therefore, be in case~(I) or~(III); either $p$ is geometrically trivial or it is nonorthogonal to the constants.
In the trivial case, $D_2$ already implies $D_m$ for all $m$, and so $p$ is completely disintegrated.
So it remains to consider the case when $p$ is nonorthogonal to the constants, and show that $D_4$ fails.

Indeed, from nonorthogonality to the constants, and using $D_2$, exactly as in the proof of Proposition~\ref{c2internal}, we conclude that $p$ is actually $\C$-internal and either interdefinable with the generic type of the constants or weakly $\C$-orthogonal.
In the former case $D_2$ would clearly fail, so $p$ is weakly $\C$-orthogonal.
It follows that $S:=p(\U)$ is a strongly minimal homogeneous space for the binding group action.
These are completely classified (see, for example,~\cite[Fact~1.6.25]{gst}), and it follows from that classification that if $\overline a=(a_1,a_2,a_3)\models p^{(3)}$ then the nonforking extension of $p$ to $A\overline a$ is not weakly orthogonal to $\C$.
This gives rise to an $A\overline a$-definable function $f:X\to\C$ where $X\subseteq S\setminus\{a_1,a_2,a_3\}$ is cofinite, 
$f(X)$ is cofinite, and the fibres of~$f$ are finite.
Let $c\in f(X)\cap\mathbb Q$ and choose any $a_4\in f^{-1}(c)$.
Then $a_4\in\acl(A\overline a)$, and hence $(a_1,a_2,a_3,a_4)$ witnesses the failure of $D_4$.
\end{proof}

The following example shows that Theorem~\ref{c3min} is sharp in the sense that $D_2$ would not suffice for minimality.

\begin{example}
\label{exa}
In~\cite[$\S4.2$]{nmdeg} the first and third authors gave an example, for each $n\geq 3$, of a $\C$-internal type $p$ of $U$-rank $n-1$ whose binding group action is that of $\PGL_n$ on $\mathbb P^{n-1}$. In particular, the binding group action is $2$-transitive, and hence $D_2$ holds for these nonminimal types.
The construction is based on one from~\cite{jin-moosa} when $n=2$.
Essentially, $p$ is the generic type of the projectivisation of the space of solutions to $\delta X=BX$ where $B$ is an $n\times n$ matrix whose entries form a differentially transcendental set, and $X$ is an $n$-column of variables.
In fact, using the solution to the inverse differenital Galois problem for connected linear algebraic groups in the non-autonomous case~\cite{mitschi-singer}, such examples will exist over arbitrary finitely generated differential fields as long as the transcendence degree over the constants is positive.
\end{example}

Let us point out that both cases of the ``moreover" clause of Theorem~\ref{c3min} actually occur.
That there are minimal completely disintegrated types is well known, with constructions appearing in~\cite{HrIt, mcgrail2000search, notmin}.
The following shows that case~(ii) also occurs: there are minimal types nonorthogonal to the constants satisfying $D_3$.

\begin{example}
\label{exb}
Consider the order~$1$ Riccati equation $\delta y=y^2+c$ with $c\in\mathbb C(t)$ a transcendental rational function.
The generic type $p$ of this equation is minimal and $\C$-internal.
Nagloo~\cite{nagloo2020algebraic} has shown that it does satisfy~$D_3$.
(Another argument is by~\cite[$\S4.3$]{nmdeg} where it is shown that the binding group acts $3$-transitively on $p$.)
\end{example}

\medskip
\subsection{The autonomous case}
When the type in Theorem~\ref{c3min} is over the constants then the theorem is not sharp.
This is because the inverse differential Galois problem fails in the autonomous case.
Indeed, we get a stronger result:

\begin{theorem}
\label{c2min}
Suppose $A\subset\C$ and $p\in S(A)$ is a stationary nonalgebraic type satisfying~$D_2$.
Then $p$ is minimal and completely disintegrated.
\end{theorem}

\begin{proof}
We first argue that $p$ must be orthogonal to the constants.
Assume, toward a contradiction, that $p$ is nonorthogonal to $\C$.
Then, by the proof of Proposition~\ref{c2internal}, $D_2$ forces $p$ to be $\C$-internal and weakly $\C$-orthogonal.
Moreover, exactly as in the proof of Theorem~\ref{c3min}, the binding group $G:=\aut_A(p/\C)$ will act outright $2$-transitively on the set $S:=p(\U)$.

The first thing we observe is that $2$-transitivity implies $G$ is centerless.
Indeed, this is an elementary argument about group actions, and is in fact the first paragraph of the proof of Knop's Satz~2 from~\cite{knop1983mehrfach} referred to above.

We deduce a contradiction by showing that for types over constants the binding group can never be centerless, unless the binding group itself is trivial.
This uses some differential Galois theory.
Let $k=\acl(A)\subset\C$, and denote also by $p$ the unique extension to $k$.
We can find a fundamental system of solutions  $a_1,\ldots, a_n \models p$ witnessing $\C$-internality of $p$ in the prime model $M$ over $k$.
This means that there exists a $k$-definable function $\phi(\overline{x}, \overline{y})$ such that every realization of $p$ is of the form $\phi(a_1,\ldots, a_n,  \overline{c})$ for some $\overline{c} \in \mathcal{C}$.
The type $q = \tp(a_1,\ldots, a_n/k)$ is also internal to the constants and weakly orthogonal to the constants and $a_1,\ldots, a_n$ generate a strongly normal extension $K := k \langle a_1,\ldots, a_n\rangle$ of $k$ in the sense of Kolchin (see \cite[$\S$3.2]{Sanchez-Pillay}).
Moreover, we have  a natural identification 
$$\aut_k(q/\mathcal C) \simeq \aut_k(p/\mathcal C) = G$$
given for $\sigma \in\aut_k(q/\mathcal C)$ by $$\sigma(\phi(a_1,\ldots, a_n,  \overline{c})) = \phi(\sigma(a_1),\ldots, \sigma(a_n), \overline{c})$$
It follows that $G$ is definably isomorphic to the set of $\C$-points of the differential Galois group $\Gal(K/k)$ of the strongly normal extension $K \mid k$.
Since $G$ has a trivial center, so does the algebraic group $\Gal(K/k)$.
By \cite[Theorem 13] {Ros}, centerless algebraic groups are linear, and hence $\Gal(K/k)$ is a linear algebraic group and $K \mid k$ is a Picard-Vessiot extension.
We claim that $\Gal(K/k)$ is commutative (this was already observed in the end of the proof of Proposition 4.9 of \cite{JJP}): indeed assume that $K$ is generated over $k$ by a fundamental system of solutions of 
$\displaystyle Y' = AY$.
After changing the basis,  we can assume that $A$ is given in Jordan's normal form.  For a Jordan block $J(\lambda,m)$ with eigenvalue $\lambda$ and size $m$,  a fundamental system of solutions is given by:
$$y_1(t) = \begin{pmatrix} e^{\lambda t} \\ 0 \\ \vdots \\  0  \end{pmatrix}, y_2(t) = \begin{pmatrix} e^{\lambda t} \\ t e^{\lambda t} \\ 0 \\ \vdots \\ 0   \end{pmatrix}, \ldots, y_r(t) =  \begin{pmatrix}e^{\lambda t} \\ te^{\lambda t} \\ \vdots \\ \frac{t^i} {i!} e^{\lambda t} \\ \vdots \\ \frac {t^r}{r!} e^{\lambda t} \end{pmatrix}.$$
Hence the entries of a fundamental system of solutions of $Y' = AY$ belong to the differential field $L = k(t, e^{\lambda_1 t}, \ldots,  e^{\lambda_s t} )$ where $\lambda_1,\ldots, \lambda_s$ are the eigenvalues of $A$. It follows that $K \mid k$ is a subextension of $ L\mid k$ and by Galois correspondence that:
$$\Gal(K/k) \simeq\Gal(L/k)/\Gal(L/K).$$
This is commutative because $\Gal(L/k) \simeq \mathbb{G}_a \times(\mathbb{G}_m)^l$ is commutative.
But that forces $\Gal(K/k)$ to be trivial, contradicting the fact that $G$ is not trivial (as for example it acts transitively on the set of realisations of the nonalgebraic type $p$).

We have thus proved that $p$ is orthogonal to the constants.
By Proposition~\ref{c2internal}, it is therefore also minimal.
Now, every minimal type over the constants and orthogonal to the constants is geometrically trivial, and hence $D_2$ implies completely disintegrated, as desired.\end{proof}

\medskip
\subsection{Deducing Theorem~\ref{A}}
Let us spell out how Theorem~\ref{c3min} implies Theorem~\ref{A} of the Introduction.
This is a standard translation from model-theoretic to differential-algebraic language.

We are working over a differential field $(k,\delta)$ of characteristic zero, with an order~$n$ algebraic differential equation
$$P(y,\delta y,\dots,\delta^{(n)}y)=0$$
where $P\in k[X_0,\dots,X_n]$ is irreducible.
We may assume the equation has nonalgebraic solutions, or there is nothing to prove.
Denote by $Y\subseteq\U$ the Kolchin closed set defined by this equation.
Note that $(C_1)$ already implies that $Y$ has no infinite Kolchin closed subsets over $k$ of order less than~$n$.
Together with irreducibility of~$P$, this implies that $Y$ is Kolchin irreducible and that all points of $Y\setminus k^{\alg}$ realise the Kolchin generic type $p\in S(k)$.
This is a nonalgebraic stationary type.
Moreover, for each $m>1$, the condition $(C_m)$ on the equation translates precisely into condition $D_m$ on $p$.

If the equation satisfies $(C_3)$ then $p$ is minimal by Theorem~\ref{c3min}.
If, in addition, $n>1$, then $p$ is orthogonal to the constants, so that we are in case~(i) of Theorem~\ref{c3min}, and $p$ is completely disintegrated. It follows that $(C_m)$ holds of the equation, for all~$m$.

On the other hand, if $\delta$ is trivial on $k$, and the equation satisfies $(C_2)$, then Theorem~\ref{c2min} implies that $p$ is minimal and totally disintegrated.
Again we conclude $(C_m)$ holds of the equation, for all~$m$.
\qed

\bigskip
\section{A formulation for $D$-varieties}
\label{sectdvar}

\noindent
There is a well studied correspondence between types in $\DCF_0$ and so called ``$D$-varieties", and it is worth reformulating Theorem~\ref{c3min} in these terms.
The idea goes back to Buium~\cite{Buium1}, who considered algebraic varieties over a differential field $(k,\delta)$ equipped with an extension of the derivation $\delta$ to the structure sheaf of the variety.
Later, an equivalent formulation, paralleling more closely the language of algebraic vector fields, came into use.
We point the reader to~\cite[$\S2$]{kowalski2006quantifier} for a more detailed survey with references.
We restrict ourselves here to a very concise explanation.

Fix a differential field $(k,\delta)$ of characteristic zero.
If $V\subseteq \mathbb A^n$ is an irreducible affine algebraic variety over $k$, then the {\em prolongation} of $V$ is the algebraic variety $\tau V\subseteq\mathbb A^{2n}$ over~$k$ whose defining equations are
\begin{eqnarray*}
P(x_1,\dots,x_n)&=&0\\
P^\delta(x_1,\dots,x_n)+\sum_{i=1}^n\frac{\partial P}{\partial x_i}(x_1,\dots,x_n)\cdot y_i&=&0
\end{eqnarray*}
for each $P$ vanishing on $V$.
Here $P^\delta$ denotes the polynomial obtained by applying $\delta$ to the coefficients of $P$.
The projection onto the first $n$ coordinates gives us a surjective morphism $\pi:\tau V\to V$.

\begin{definition}
\label{defn-dvar}
A {\em $D$-variety over $k$} is a pair $(V,s)$ where $V$ is an irreducible algebraic variety over $k$ and $s:V\to\tau V$ is a regular section to the prolongation  defined over $k$.
A {\em $D$-subvariety} of $(V,s)$ is then a $D$-variety $(W,t)$ where $W$ is a closed subvariety of $V$ over $k$ and $t=s|_W$.
\end{definition}

Note that if $\delta=0$ on $k$, then the above equations for $\tau V$ reduce to the familiar equations for the tangent bundle $TV$.
A $D$-variety in that case is nothing other than an algebraic vector field, and a $D$-subvariety is an invariant subvariety of the vector field.
In general, $\tau V$ will be a torsor for $TV$; for each $a\in V$ the fibre $\tau_a V$ is an affine translate of the tangent space $T_aV$.

Associated to a $D$-variety $(V,s)$ over $k$ is a certain complete type over $k$ in $\DCF_0$, namely the {\em generic type of $(V,s)$ over $k$}.
This type, $p(x)$, is axiomatised by asserting that $x$ is Zariski-generic in $V$ over $k$, and that the following system of order $1$ algebraic differential equations holds: $s(x)=(x,\delta(x))$.
In fact, up to interdefinability, all finite rank types in $\DCF_0$ arise in this way; as the generic types of $D$-varieties.
Results about finite rank types in $\DCF_0$, such as Theorem~\ref{c3min} above, can therefore be translated into algebro-geometric statements about $D$-varieties.
For example, the condition we called $D_m$ in the previous section, namely that of having every distinct $m$-tuple of realisations being independent, when applied to the generic type of $(V,s)$ translates to the absence of any proper $D$-subvarieties of $(V^m,s^m)$ projecting dominantly onto each co-ordinate, other than the diagonals.
In particular, we obtain

\begin{corollary}
\label{dvarcor}
Suppose $(V,s)$ is a $D$-variety over $(k,\delta)$.
\begin{itemize}
\item[(a)]
If $\dim V>1$ and the third cartesian power $(V^3,s^3)$ admits no proper $D$-subvarieties over $k$ projecting dominantly onto each co-ordinate, other than the diagonals, then the same holds of $(V^m,s^m)$ for all $m$.
\item[(b)]
If $\delta$ is trivial on $k$ and $(V^2,s^2)$ admits no proper $D$-subvarieties over $k$ projecting dominantly onto each co-ordinate, other than the diagonal, then, for all $m$, $(V^m,s^m)$ admits no proper $D$-subvarieties over $k$ projecting dominantly onto each co-ordinate other than the diagonals.
\end{itemize}
\end{corollary}

\begin{proof}
Let $p\in S(k)$ be the generic type of $(V,s)$ over $k$.
The assumption on $(V^3,s^3)$ in part~(a) tells us that $p$ satisfies $D_3$.
Hence, by Theorem~\ref{c3min}, $p$ is minimal.
Now, the fact that $\dim V>1$ rules out the possibility of $p$ being nonorthogonal to the constants.
So $p$ is completely disintegrated.
This implies the desired conclusion about $D$-subvarieties of $(V^m,s^m)$ over $k$.
Similarly, applying Theorem~\ref{c2min} to $p$ yields part~(b).
\end{proof}

\begin{remark}
In the non-autonomous case, if $\dim V=1$, then one has to make the assumption on $D$-subvarieties of $(V^4,s^4)$ in order to get $D_4$ on $p$ and thus conclude from Theorem~\ref{c3min} that $p$ is completely disintegrated.
\end{remark}

Note that Corollary~\ref{dvarcor}(b) is Theorem~\ref{B} of the Introduction.

\bigskip
\section{The case of compact complex manifolds}

\noindent
As mentioned in the introduction, the main part of Theorem~\ref{c3min} holds also in $\ccm$, the theory of compact complex manifolds.
This translates into the following statement in bimeromorphic geometry that may be of independent interest:

\begin{theorem}
\label{ccm3min}
Suppose $f:X\to Y$ is a fibre space of compact K\"ahler manifolds\footnote{More generally our results apply to  compact complex anlaytic varieties in Fujiki's class $\mathscr C$. In fact, all that is needed is that they be {\em essentially saturated} in the sense of~\cite{saturated}.}.
Suppose that the three-fold fibred product $X\times_YX\times_YX$ contains no proper complex analytic subvarieties that projects onto each co-ordinate, other than the diagonals.
Then the `general' fibre of $f$ is simple.
\end{theorem}

Let us define the terminology used here.
That $f$ is a {\em fibre space} means that off a proper complex analytic subset of~$Y$ the fibres are irreducible.
A property holds of the {\em `general'} fibre of $f$ if it holds of all fibres off a countable union of proper complex analytic subsets of $Y$.
Finally, we mean {\em simple} in the sense of~\cite{fujiki83}; namely, not covered by an analytic family of proper positive-dimensional analytic subsets.

We will not give a proper proof of this theorem, as it is really the same as (part of) Theorem~\ref{c3min}.
But let us give some explanations.
First of all, we work in $\ccm$ which is the theory of the multisorted structure $\mathcal A$ where there is a sort for each compact complex analytic variety and a predicate for every complex analytic subset of a cartesian product of sorts.
The role of the constants is played by the projective line which is a sort on which the induced structure is bi-interpretable with that of a pure algebraically closed field.
The truth of the Zilber dichotomy in $\ccm$ tells us that every non locally modular minimal type is nonorthogonal to the projective line.
We suggest~\cite{moosapillay-survey} for an introduction to, and survey of, the subject.

Theorem~\ref{ccm3min} is proved by considering the generic type~$p$ of the generic fibre of~$f$, observing that the assumption on the three-fold fibred product expresses precisely that $p$ satisfies $D_3$, and then deducing, exactly as in the proof of Theorem~\ref{c3min}, that~$p$ is minimal.
But for a compact complex analytic variety to be simple is equivalent to its generic type in $\ccm$ being minimal.

But there is a subtlety: the structure $\mathcal A$ is not saturated and so ``generic" in the above paragraph has to be understood, {\em a priori}, in the sense of a nonstandard elementary extension of $\mathcal A$.
See~\cite[$\S2$]{ret} for details about the Zariski geometry in nonstandard models.
The reason for restricting to the K\"ahler case is then precisely so that we can, following~\cite{saturated}, find a suitable countable language in which the standard model is saturated.
We can thus conclude simplicity of the `general' fibre in $\mathcal A$ from simplicity of the generic fibre in an elementary extension.

Note that we only concluded minimality, and did not state an analogue of the ``moreover" clause in Theorem~\ref{c3min}.
The reason for this is that there remains some ambiguity about the characterisation of nontrivial locally modular types in $\ccm$.
It is known that they are nonorthogonal to simple complex tori, but it seems as yet unclear whether one has the same control over parameters that one does in $\DCF_0$.
This is related to Question~5.1 of~\cite{isolated}.


\bigskip

\begin{thebibliography}{10}

\bibitem{BC2008}
Alexandre Borovik and Gregory Cherlin.
\newblock Permutation groups of finite {M}orley rank.
\newblock In {\em Model theory with applications to algebra and analysis.
  {V}ol. 2}, volume 350 of {\em London Math. Soc. Lecture Note Ser.}, pages
  59--124. Cambridge Univ. Press, Cambridge, 2008.

\bibitem{Buium1}
A.~Buium.
\newblock {\em Differential algebraic groups of finite dimension}.
\newblock Lecture Notes in Math., 1506, Springer-Verlag, 1992.

\bibitem{casale2020ax}
Guy Casale, James Freitag, and Joel Nagloo.
\newblock Ax-{L}indemann-{W}eierstrass with derivatives and the genus 0
  {F}uchsian groups.
\newblock {\em Ann. of Math. (2)}, 192(3):721--765, 2020.

\bibitem{dreyfus2018hypertranscendence}
Thomas Dreyfus, Charlotte Hardouin, and Julien Roques.
\newblock Hypertranscendence of solutions of {M}ahler equations.
\newblock {\em J. Eur. Math. Soc. (JEMS)}, 20(9):2209--2238, 2018.

\bibitem{nmdeg}
James Freitag and Rahim Moosa.
\newblock Bounding nonminimality and a conjecture of {B}orovik-{C}herlin.
\newblock Preprint, arXiv:2106.02537.

\bibitem{freitag2017strong}
James Freitag and Thomas Scanlon.
\newblock Strong minimality and the $ j $-function.
\newblock {\em Journal of the European Mathematical Society}, 20(1):119--136,
  2017.

\bibitem{fujiki83}
Akira Fujiki.
\newblock On the structure of compact complex manifolds in {C}.
\newblock In {\em Algebraic varieties and analytic varieties ({T}okyo, 1981)},
  volume~1 of {\em Adv. Stud. Pure Math.}, pages 231--302. North-Holland,
  Amsterdam, 1983.

\bibitem{HrIt}
Ehud Hrushovski and Masanori Itai.
\newblock On model complete differential fields.
\newblock {\em Transactions of the American Mathematical Society}, Volume
  355(11):4267--4296, 2003.

\bibitem{HrSo}
Ehud Hrushovski and \v{Z}eljko Sokolovi\'{c}.
\newblock Strongly minimal sets in differentially closed fields.
\newblock {\em unpublished manuscript}, 1993.

\bibitem{jaoui2019generic}
R{\'e}mi Jaoui.
\newblock Generic planar algebraic vector fields are disintegrated.
\newblock {\em arXiv preprint arXiv:1905.09429}, 2019.

\bibitem{JJP}
R\'semi Jaoui, L\'eo Jimenez, and Anand Pillay.
\newblock Relative internality and definable fibrations.
\newblock Preprint, arXiv:2009.06014.

\bibitem{jin-moosa}
Ruizhang Jin and Rahim Moosa.
\newblock Internality of logarithmic-differential pullbacks.
\newblock {\em Trans. Amer. Math. Soc.}, 373(7):4863--4887, 2020.

\bibitem{klingler2018bi}
B.~Klingler, E.~Ullmo, and A.~Yafaev.
\newblock Bi-algebraic geometry and the {A}ndr\'{e}-{O}ort conjecture.
\newblock In {\em Algebraic geometry: {S}alt {L}ake {C}ity 2015}, volume~97 of
  {\em Proc. Sympos. Pure Math.}, pages 319--359. Amer. Math. Soc., Providence,
  RI, 2018.

\bibitem{knop1983mehrfach}
Friedrich Knop.
\newblock Mehrfach transitive operationen algebraischer gruppen.
\newblock {\em Archiv der Mathematik}, 41(5):438--446, 1983.

\bibitem{kolchin1968algebraic}
Ellis~R Kolchin.
\newblock Algebraic groups and algebraic dependence.
\newblock {\em American Journal of Mathematics}, pages 1151--1164, 1968.

\bibitem{kowalski2006quantifier}
Piotr Kowalski and Anand Pillay.
\newblock Quantifier elimination for algebraic {$D$}-groups.
\newblock {\em Trans. Amer. Math. Soc.}, 358(1):167--181, 2006.

\bibitem{isolated}
Omar Le\'{o}n~S\'{a}nchez and Rahim Moosa.
\newblock Isolated types of finite rank: an abstract {D}ixmier-{M}oeglin
  equivalence.
\newblock {\em Selecta Math. (N.S.)}, 25(1):Paper No. 10, 10, 2019.

\bibitem{Sanchez-Pillay}
Omar Le\'{o}n~S\'{a}nchez and Anand Pillay.
\newblock Some definable {G}alois theory and examples.
\newblock {\em Bull. Symb. Log.}, 23(2):145--159, 2017.

\bibitem{Marker96modeltheory}
David Marker.
\newblock Model theory of differentiable fields.
\newblock In {\em Lecture Notes in Logic 5}. Springer, 1996.

\bibitem{mcgrail2000search}
Tracey McGrail.
\newblock The search for trivial types.
\newblock {\em Illinois Journal of Mathematics}, 44(2):263--271, 2000.

\bibitem{mitschi-singer}
C.~Mitschi and M.~F. Singer.
\newblock Connected linear groups as differential {G}alois groups.
\newblock {\em J. Algebra}, 184(1):333--361, 1996.

\bibitem{ret}
Rahim Moosa.
\newblock A nonstandard {R}iemann existence theorem.
\newblock {\em Trans. Amer. Math. Soc.}, 356(5):1781--1797, 2004.

\bibitem{saturated}
Rahim Moosa.
\newblock On saturation and the model theory of compact {K}\"{a}hler manifolds.
\newblock {\em J. Reine Angew. Math.}, 586:1--20, 2005.

\bibitem{moosapillay-survey}
Rahim Moosa and Anand Pillay.
\newblock Model theory and {K}\"{a}hler geometry.
\newblock In {\em Model theory with applications to algebra and analysis.
  {V}ol. 1}, volume 349 of {\em London Math. Soc. Lecture Note Ser.}, pages
  167--195. Cambridge Univ. Press, Cambridge, 2008.

\bibitem{moosa2014some}
Rahim Moosa and Anand Pillay.
\newblock Some model theory of fibrations and algebraic reductions.
\newblock {\em Selecta Mathematica}, 20(4):1067--1082, 2014.

\bibitem{nagloo2020algebraic}
Joel Nagloo.
\newblock Algebraic independence of generic painlev{\'e} transcendents: {PIII}
  and {PVI}.
\newblock {\em Bulletin of the London Mathematical Society}, 52(1):100--108,
  2020.

\bibitem{nagloo2014algebraic}
Joel Nagloo and Anand Pillay.
\newblock On the algebraic independence of generic painlev{\'e} transcendents.
\newblock {\em Compositio Mathematica}, 150(04):668--678, 2014.

\bibitem{nagloo2017algebraic}
Joel Nagloo and Anand Pillay.
\newblock On algebraic relations between solutions of a generic {P}ainlev\'{e}
  equation.
\newblock {\em J. Reine Angew. Math.}, 726:1--27, 2017.

\bibitem{pila2016ax}
Jonathan Pila and Jacob Tsimerman.
\newblock Ax-{S}chanuel for the {$j$}-function.
\newblock {\em Duke Math. J.}, 165(13):2587--2605, 2016.

\bibitem{gst}
Anand Pillay.
\newblock {\em Geometric Stability Theory}.
\newblock Oxford University Press, 1996.

\bibitem{Ros}
Maxwell Rosenlicht.
\newblock Some basic theorems on algebraic groups.
\newblock {\em Amer. J. Math.}, 78:401--443, 1956.

\bibitem{notmin}
Maxwell Rosenlicht.
\newblock The nonminimality of the differential closure.
\newblock {\em Pacific J. Math.}, 52:529 -- 537, 1974.

\bibitem{singer1988algebraic}
Michael~F. Singer.
\newblock Algebraic relations among solutions of linear differential equations:
  {F}ano's theorem.
\newblock {\em Amer. J. Math.}, 110(1):115--143, 1988.

\end{thebibliography}

\end{document}